\documentclass[12pt]{amsart}
\usepackage{geometry}
\usepackage[colorlinks,citecolor = red, linkcolor=blue,hyperindex]{hyperref}
\usepackage{euscript,eufrak,verbatim, mathrsfs}
\usepackage[psamsfonts]{amssymb}
\usepackage{bbm}
\usepackage{graphicx}
\usepackage{float, tikz}
\usepackage{pgfplots}
\pgfplotsset{compat=1.18}
\usepgfplotslibrary{fillbetween}
\usepackage{caption}
\usepackage{subcaption}

\usepackage{extarrows}
\usepackage[all, cmtip]{xy}

\usepackage{upref, xcolor, dsfont}
\usepackage{amsfonts,amsmath,amstext,amsbsy, amsopn,amsthm}
\usepackage{enumerate}

\usepackage{url}
\usepackage{mathtools}
\usepackage{bookmark}

\usepackage{euscript}
\usepackage{helvet}
\usepackage{courier}
\usepackage{type1cm}

\usepackage{multicol}
\usepackage[bottom]{footmisc}

\newtheorem{theorem}{Theorem}[section]
\newtheorem*{theorem*}{Theorem A}
\newtheorem{lemma}[theorem]{Lemma}

\newtheorem*{definition*}{Definition}
\newtheorem*{remark*}{Remark}

\newtheorem*{observation*}{Observation}

\newtheorem*{assumption*}{Assumption}

\theoremstyle{definition}

\newtheorem{question}{Question}
\newtheorem*{problem*}{Problem}

\theoremstyle{remark}
\newtheorem{remark}{Remark}[section]

\newcommand{\D}{\mathbb{D}}

\newcommand{\Log}{\mathrm{Log}}

\newcommand{\dd}{\mathrm{d}}

\begin{document}

\title{On the converse of the Shimorin--Pel\'aez--R\"atty\"a--Wick theorem}

\author{Yuerang Li}
\address{Yuerang Li: College of Mathematics and Statistics, Chongqing University, Chongqing 401331, China}
\email{yuerangli@outlook.com}

\author{Zipeng Wang}
\address{Zipeng Wang: College of Mathematics and Statistics, Chongqing University, Chongqing 401331, China}
\email{zipengwang2012@gmail.com; zipengwang@cqu.edu.cn}


\maketitle

\begin{abstract}
We establish a converse of the Shimorin--Pel\'{a}ez--R\"{a}tty\"{a}--Wick theorem.
Specifically, we obtain necessary and sufficient conditions for a Shimorin kernel to be the kernel of a radial, logarithmically subharmonic weighted Bergman space.
\end{abstract}

\section{Introduction and main results}

For a nonnegative measurable function $\sigma$ on the unit disk $\D$, the weighted Bergman space $L_a^2(\sigma)$ consists of all holomorphic functions on $\D$ that are square-integrable with respect to $\sigma \dd A$. That is,
\[
L_a^2(\sigma) := \left\{ f \in \mathrm{Hol}(\D) : \|f\|_{L_a^2(\sigma)}^2 = \int_{\D} |f(z)|^2 \sigma(z) \dd A(z) < \infty \right\},
\]
where $\dd A(z) = \frac{dx dy}{\pi}$ denotes the normalized area measure on the complex plane.

The space $L_a^2(\sigma)$ is a reproducing kernel Hilbert space if it is a Hilbert space and the point evaluation functionals at all points of the unit disk are bounded. In that case, for each $z \in \mathbb{D}$, there exists a unique element $K_{\sigma}(\cdot,z) \in L_a^2(\sigma)$ such that
\[
f(z) = \int_{\mathbb{D}} f(\lambda) \overline{K_{\sigma}(\lambda,z)} \, \sigma(\lambda) \, \dd A(\lambda)
\]
for all $f \in L_a^2(\sigma)$. The function $K_{\sigma}(\cdot,z)$ is the reproducing kernel, and it is called the Bergman kernel of $L_a^2(\sigma)$ at the point $z$.

A particularly important class of weighted Bergman spaces is the one induced by a logarithmically subharmonic weight that reproduces the origin; that is, those $\sigma$ for which $\log\sigma$ is subharmonic on $\D$ and
$
p(0) = \int_{\D} p(z) \sigma(z) \dd A(z)
$
for every polynomial $p$. This weight induces a natural hyperbolic Riemannian metric on the unit disk and also serves as a starting point for proving the positivity of the Green function for the weighted biharmonic operator (see \cite{Shi02} for more details).

For a logarithmically subharmonic weight, the Shimorin problem in kernel representation theory asks whether its reproducing kernel is of Bergman-type, i.e., whether the kernel admits a Bergman-type integral representation. In \cite{Shi02}, Shimorin proved that every radial logarithmically subharmonic weight admits such a representation. He proved that for a radial logarithmically subharmonic weight, there exists a measure $\nu$ on $[0,1]$ such that its kernel $K_{\sigma}$ admits a Bergman-type integral representation
\[
K_{\sigma}(z,\lambda) = \frac{1}{1 - z\bar{\lambda}} \int_{0}^{1} \frac{1}{1 - rz\bar{\lambda}} \, d\nu(r), \qquad z, \lambda \in \mathbb{D}.
\]

Let $\nu$ be a positive Borel measure on $[0,1]$. We call the Bergman-type integral
\[
S_\nu(z,\lambda)=\frac{1}{1 - z\bar{\lambda}} \int_{0}^{1} \frac{1}{1 - rz\bar{\lambda}} \, \dd \nu(r)
\]
a Shimorin-type kernel on the unit disk.

Recall that a weight $\sigma$ on the unit disk is radial if $\sigma(z) = \sigma(|z|)$ for all $z \in \mathbb{D}$.
Among radial weights, an important class is that of $\widehat{\mathcal{D}}$-weight. For example, this class characterizes those radial weights for which the weighted Bergman projection is bounded from $L^\infty$ onto the Bloch space. We recall that a finite positive Borel measure $\omega$ on $[0,1)$ is said to belong to
$\widehat{\mathcal{D}}$ if there exists a constant $C>0$ such that
\[
\widehat{\omega}(r)
:=
\omega([r,1))
\le
C \widehat{\omega} \left(\frac{1+r}{2}\right),
\qquad r\in[0,1).
\]
One can consult \cite{PR14}, \cite{PR15}, \cite{PR21} and \cite{PRW19} for more details.

In 2019, Peláez, Rättyä, and Wick \cite{PRW19} systematically developed a Shimorin-type integral representation for Bergman kernels. In particular, they established a surprising and interesting connection between logarithmically subharmonic weights and $\widehat{\mathcal{D}}$-weights. Precisely, they showed that for any finite measure $\nu$ on $[0,1]$ satisfying
\begin{equation}\label{eq:prw}
\int_0^1 \frac{\dd \nu(r)}{1-r} = \infty,
\end{equation}
the Shimorin-type kernel is the reproducing kernel of a $\widehat{\mathcal{D}}$-weighted Bergman space.

\vspace{0.2cm}

\noindent \textbf{Shimorin–Peláez–Rättyä–Wick Theorem.} Every radial logarithmically subharmonic weight is a $\widehat{\mathcal{D}}$-weight, and its reproducing kernel has a Shimorin-type representation.

\vspace{0.2cm}

From its definition, we only know that a Shimorin-type kernel is positive definite; it is not necessarily a Bergman kernel. For example, when $\nu = \delta_0$, the Dirac measure at $0$, we have
\[
S_{\nu}(z,\lambda) = \frac{1}{1 - z\bar{\lambda}},
\]
which is the reproducing kernel of the Hardy space. It is therefore natural to ask when a Shimorin-type kernel represents a Bergman kernel. Moreover, this occurs when the weight is logarithmically subharmonic. Precisely, we study the following question concerning the converse of the Shimorin--Peláez--Rättyä--Wick theorem.

\begin{question}\label{pro:key}
Let $\nu$ be a positive Borel measure on $[0,1]$ and let $S_\nu$ denote the associated Shimorin-type kernel on the unit disk. Under what conditions is $S_\nu$ the kernel of a radial, logarithmically subharmonic weighted Bergman space?
\end{question}

The following result, essentially due to Shimorin, Peláez, Rättyä, and Wick, characterizes when a Shimorin-type kernel is a Bergman kernel.

\begin{theorem}\label{representation of K_v}
Let $\nu$ be a finite positive Borel measure on $[0,1]$. The Shimorin-type kernel $S_\nu$ is a Bergman kernel if and only if the measure $\nu$ satisfies the Peláez–Rättyä–Wick condition, i.e.,
\[
\int_{0}^{1} \frac{1}{1 - r} \dd \nu(r) = \infty.
\]
\end{theorem}

We now state our main result, which answers Question~\ref{pro:key}.

\begin{theorem}\label{thm:shimorin-kernel-characterization}
Let $\nu$ be a finite positive Borel measure on $[0,1]$ satisfying the Peláez–Rättyä–Wick condition.
The Shimorin-type kernel $S_\nu$ is the kernel of a radial logarithmically subharmonic weighted Bergman space if and only if there exists a logarithmically convex function
$h:(0,\infty)\to(0,\infty)$
such that
\[
\left(\int_0^1 \frac{1 - r^{n + 1}}{1 - r} \dd \nu(r)\right)
\left(\int_0^\infty e^{-nt} h(t) \dd t\right)=1,\qquad n\in\mathbb N,
\]
and
\[
\sup_{t\ge T} e^{t} h(t) < \infty
\]
for some constant $T>0$.
\end{theorem}

\medskip
{\flushleft\bf Acknowledgements.} This work is supported by the National Natural Science Foundation of China (No.12471116) and 2025CDJ-IAIS YB-004 (Chongqing University).

\section{The proof of Theorem \ref{representation of K_v}}
The main goal of this section is to prove Theorem \ref{representation of K_v}, which follows from Theorem \ref{new-representation of K_v}. Note that a radial weight on the unit disk $\D$ is determined by a measure on $[0,1)$. More precisely, let $\omega$ be a finite positive Borel measure on $[0,1)$.
The associated radial measure $\omega\otimes m$ on $\D$ is defined by
\[
\dd(\omega\otimes m)(re^{i\theta})
=
r \dd \omega(r) \frac{\dd \theta}{\pi},
\]
where $m=\frac{\dd \theta}{\pi}$ is the normalized Haar measure on the unit circle $\mathbb{T}$. Then the weighted Bergman space is
\[
L_a^2(\omega) := \left\{ f\in \mathrm{Hol}(\D) : \|f\|_{L_a^2(\omega)}^2 = \int_{\D} |f(z)|^2 \dd(\omega\otimes m)(z)<\infty \right\}.
\]

\begin{theorem}\label{new-representation of K_v}
  Let $\nu$ be a finite positive Borel measure on $[0,1]$, and let $S_{\nu}$ be the corresponding Shimorin-type kernel.
There exists a finite positive Borel measure $\omega$ on $[0,1)$ with $\omega([r,1)) > 0$ for all $0 < r < 1$ such that
  \[
  K_{\omega}(z,\lambda) = S_{\nu}(z,\lambda), \quad z, \lambda \in \D,
  \]
  if and only if
  \[
  \int_{0}^{1} \frac{1}{1 - r}\dd \nu(r) = \infty.
  \]
Moreover, $\omega$ is a $\widehat{\mathcal{D}}$-weight.
\end{theorem}

\begin{remark}
We emphasize once again that this result is due to Shimorin, Peláez, Rättyä, and Wick. Based on recent work by Berg and Durán \cite{BD}, we present a self-contained proof that is of independent interest.
\end{remark}

\begin{remark}
  For any radial weight $\omega$, the Bergman space $L^2_a(\omega)$ is a reproducing kernel Hilbert space if and only if $\omega([r,1)) > 0$ for all $0 < r < 1$.
\end{remark}

We begin with the following two useful results of Berg and Durán.
\begin{lemma}[Theorem 1.1 in \cite{BD}]\label{existence of harmonic Hausdorff moment sequence}
Let $(a_n)_n$ be a Hausdorff moment sequence of a nonzero measure $\nu$. Then
the sequence $(b_n)_n$ defined by
\[
b_n=\frac{1}{a_0+\cdots+a_n}
\]
is a Hausdorff moment sequence, and its associated measure
$
\mu
$
satisfies $\mu(\{0\})=0$ and
\begin{equation}\label{eq:1.1}
\left(\int_0^1 \frac{1-t^{z+1}}{1-t} \dd \nu(t)\right) \left(\int_0^1 t^z \dd \mu(t)\right) = 1, \qquad \text{for } \operatorname{Re} z \ge 0.
\end{equation}
\end{lemma}

\begin{lemma}[Theorem 4.1 in \cite{BD}]\label{representation of harmonic Hausdorff moment sequence}
Let $\nu$ be a nonzero measure on $[0,1]$ and write
$
\nu = a\delta_0 + b\delta_1 + \widetilde{\nu}
$
with $a=\nu(\{0\}),  b=\nu(\{1\}), \widetilde{\nu}=\nu|_{(0,1)}.$
Let $f$ be the nonzero Bernstein function given by
\[
f(s) = \int_0^1 \frac{1 - t^s }{1 - t} \dd \nu(t) = a + bs + \int_{0}^{1} \frac{1 - t^s}{1 - t} \dd \widetilde{\nu}(t), \quad s > 0.
\]
Then the measure
$
\mu
$
determined by Lemma \ref{existence of harmonic Hausdorff moment sequence}
is given by
$
\dd \mu(x) = x \dd \widetilde{\kappa}(x),
$
where $\widetilde{\kappa}$ is the potential kernel of the corresponding product convolution semigroup satisfying
\[
\int_0^1 t^s \dd \widetilde{\kappa}(t) = \frac{1}{f(s)}, \quad s > 0.
\]
Moreover, the moment sequence
\[
b_n=\frac{1}{a_0+\cdots+a_n}
\]
of $\mu$ is given by
\[
b_n=\frac{1}{f(n+1)}, \qquad n\ge 0.
\]
\end{lemma}

\begin{proof}[Proof of Theorem \ref{new-representation of K_v}]
  Let $\nu$ be a finite positive Borel measure on $[0,1]$, and let $S_{\nu}$ be the corresponding Shimorin-type kernel.
  Suppose that $$S_{\nu}(z,\lambda) = K_{\omega}(z,\lambda), \quad z, \lambda \in \D,$$ for some measure $\omega$ on $[0,1)$ with $\omega([r,1)) > 0$ for all $0 < r < 1$.
  Note that
  \[
  S_{\nu}(z,\lambda) = \sum_{n = 0}^{\infty} \left(\int_0^1 \frac{1 - r^{n + 1}}{1 - r} \dd \nu(r)\right) (z\bar{\lambda})^n,
  \]
  and
  \[
  K_{\omega}(z,\lambda) = \sum_{n = 0}^{\infty} \frac{1}{2\omega_{2n + 1}} (z\bar{\lambda})^n,
  \]
  where $\omega_{2n + 1} = \int_0^1 r^{2n + 1} \dd \omega(r)$.
  Then we have
  \[
  \int_{0}^{1} \frac{1 - r^{n + 1}}{1 - r}\dd \nu(r) = \frac{1}{2 \omega_{2n + 1}}.
  \]
   By the Lebesgue dominated convergence theorem,
  \[
  \lim_{n \to \infty} \omega_{2n + 1} = \lim_{n \to \infty} \int_0^1 r^{2n + 1} \dd \omega(r) =  0.
  \]
  Hence, by the monotone convergence theorem,
  \[
  \int_{0}^{1} \frac{1}{1 - r} \dd \nu(r) = \lim_{n \to \infty} \int_0^1 \frac{1 - r^{n + 1}}{1 - r}\dd \nu(r) = \lim_{n \to \infty} \frac{1}{2 \omega_{2n + 1}} = \infty.
  \]

  Conversely, consider the Hausdorff moment sequence of the measure $\nu$
  \[
  v_k = \int_0^1 r^k \dd \nu(r), \quad k \in \mathbb{N},
  \]
  by Lemma \ref{existence of harmonic Hausdorff moment sequence}, there exists a finite positive Borel measure $\mu$ on $[0,1]$ such that
  \[
  \int_0^1 \frac{1 - r^{n + 1}}{1 - r}\dd \nu(r) \cdot \int_0^1 r^n \dd \mu(r) = 1 , \quad n \in \mathbb{N}.
  \]
  Let $\psi(x) = \sqrt{x}$. By a change of variables,
  \[
  \int_{0}^{1} r^n \dd \mu(r) = \int_{0}^{1} t^{2n} \dd (\psi_{*}\mu)(t) = \int_0^1 t^{2n + 1} \frac{1}{t} \dd (\psi_{*} \mu)(t).
  \]
  Then we can define a measure $\omega$ on $[0,1]$ by
  \begin{equation}\label{e:measuremu}
  \omega(E) =  \frac{1}{2} \int_{E} \frac{1}{t} \dd (\psi_{*}\mu)(t) =  \frac{1}{2}\int_{\psi^{-1}(E)}r^{-1 / 2} \dd \mu(r),
  \end{equation}
  where $E\subset[0,1]$ is a Borel set. Hence we have
  \[
  2 \omega_{2n + 1} = \int_0^1 t^{2n + 1} \frac{1}{t} \dd (\psi_{*}\mu)(t) = \int_0^1 r^n \dd \mu(r) = \frac{1}{\int_0^1 \frac{1 - r^{n + 1}}{1 - r}\dd \nu(r)}.
  \]
To complete the proof, it remains to show that
  \begin{itemize}
    \item $\omega$ is a finite measure on $[0,1]$, i.e., $\int_{0}^{1} r^{- 1/ 2} \dd \mu(r) < \infty$;
    \item $\omega(\{1\}) = 0$, that is, $\omega$ is concentrated on $[0,1)$;
    \item For all $0 < r < 1$, $\omega([r,1)) > 0$;
    \item $\omega$ is a $\widehat{\mathcal{D}}$-weight.
  \end{itemize}
  Let $f$ be the nonzero Bernstein function given by
  \[
  f(s) = \int_0^1 \frac{1 - t^{s}}{1 - t} \dd \nu(t).
  \]
  By Lemma \ref{representation of harmonic Hausdorff moment sequence}, there is a $\sigma$-finite positive Borel measure $\widetilde{\kappa}$ on $[0,1]$ such that
  \[
  \dd \mu(t) = t\dd \widetilde{\kappa}(t), \quad \text{and} \quad \int_{0}^{1} t^s \dd \widetilde{\kappa}(t) = \frac{1}{f(s)}, \ s > 0.
  \]
  Hence,
  \[
  \int_0^1 r^{-1 / 2}\dd \mu(r) = \int_0^1 t^{1/2} \dd \widetilde{\kappa}(t) = \frac{1}{f(1/2)} < \infty.
  \]
  Next, to show $\omega(\{1\}) = 0$, we only need to show $\mu(\{1\}) = 0$ (Indeed, $\omega(\{1\}) = \frac{1}{2} \mu(\{1\})$). If $\mu(\{1\}) \ne 0$, then
  \[
  \int_0^1 r^{n} \dd \mu(r) > \mu(\{1\}) > 0.
  \]
  Since $\int_{0}^{1} \frac{1}{1 - r}\dd \nu(r) = \infty$, we have
  \[\lim_{n\to\infty}\int_{0}^{1} \frac{1 - r^{n + 1}}{1 - r} \dd \nu(r) = \infty.\]
  This contradicts the fact that
  \[
  \int_0^1 \frac{1 - r^{n + 1}}{1 - r}\dd \nu(r) \int_{0}^{1} r^n \dd \mu(r) = 1.
  \]
  Finally, we show that for all $0 < r < 1$, \[\mu([r,1)) = \int_{r}^1 \dd \mu(t) > 0.\]
  Otherwise, if there exists $r_0 \in (0,1)$ such that $\mu([r_0,1)) = 0$, then
  \[
  \int_{0}^{1} r^{n} \dd \mu(r) = \int_{0}^{r_0}r^n \dd \mu(r) \le r_0^n \mu([0,1)),
  \]
  and note that
  \[
  \int_0^1 \frac{1 - r^{n + 1}}{1 - r}\dd \nu(r) = \sum_{k = 0}^{n} \int_0^1 r^k \dd \nu(r) \le (n + 1) \nu([0,1]).
  \]
  Hence,
  \[
  \int_0^1 \frac{1 - r^{n + 1}}{1 - r}\dd \nu(r) \int_{0}^{1} r^n \dd \mu(r) \le (n + 1)r_0^{n} \mu([0,1)) \nu([0,1]).
  \]
  The right-hand side of the inequality converges to zero as $n \to \infty$, which is a contradiction.

Hence,
\[
\omega([r,1)) = \frac{1}{2} \int_{[r,1)} t^{-1} \dd (\psi_*\mu)(t) = \frac{1}{2} \int_{[r^2,1)} s^{-1/2} \dd \mu(s) \geq \frac{1}{2} \mu([r^2,1)) > 0.
\]
It remains to show that $\omega$ is a $\widehat{\mathcal{D}}$-weight.
This is equivalent to showing \cite{PRW19} that there exists a finite positive constant $C$ such that $$\omega_n \le C \omega_{2n}$$ for all $n \in \mathbb{N}$, where $\omega_n = \int_0^1 r^n \dd \omega(r)$.

The case $n=0$ is trivial. Let $n\ge 1$. By the definition \eqref{e:measuremu} of the measure $\omega$ and a change of variables, we have
\[
\omega_n = \int_0^1 t^n \dd \omega(t) =\frac{1}{2}\int_0^1 t^{n-1} \dd (\psi_*\mu)(t) =\frac{1}{2}\int_0^1 t^{(n-1)/2} \dd \mu(t).
\]
Here $\psi(x) = \sqrt{x}$.
By Lemma \ref{representation of harmonic Hausdorff moment sequence}, for every $s\ge 0$,
\[
\int_0^1 t^s \dd \mu(t)=\frac{1}{f(s+1)}.
\]
Therefore, for every $n\ge 1$,
\[
\omega_n=\frac{1}{2f \left(\frac{n+1}{2}\right)},
\qquad
\omega_{2n}=\frac{1}{2f \left(n+\frac{1}{2}\right)}.
\]
Thus
\[
\frac{\omega_n}{\omega_{2n}}
=
\frac{f \left(n+\frac{1}{2}\right)}{f \left(\frac{n+1}{2}\right)}.
\]
It remains to prove that
\[
f \left(n+\frac{1}{2}\right)\le 2 f \left(\frac{n+1}{2}\right),
\qquad n\ge 1.
\]

For $k\ge 0$, write $\nu_k:=\int_0^1 r^k \dd \nu(r).$
Since $\nu$ is a finite positive measure on $[0,1]$, the sequence $(\nu_k)_{k\ge 0}$ is decreasing, and $f$ is increasing. We now distinguish two cases.

\medskip
\noindent\textbf{Case 1: $n=2m-1$ for $m\ge 1$.}
Then
\[
f \left(n+\frac{1}{2}\right)= f(2m - \frac{1}{2}) \le f(2m)=\sum_{k=0}^{2m-1}\nu_k,
\qquad
f \left(\frac{n+1}{2}\right)=f(m)=\sum_{k=0}^{m-1}\nu_k.
\]
Since $(\nu_k)$ is decreasing, we have
$$ \sum_{k=m}^{2m-1}\nu_k \le \sum_{k=0}^{m-1}\nu_k.$$
Hence
\[
f(2m)=\sum_{k=0}^{m-1}\nu_k+\sum_{k=m}^{2m-1}\nu_k
\le 2\sum_{k=0}^{m-1}\nu_k
=2f(m).
\]
Therefore,
\[
f \left(n+\frac{1}{2}\right)\le 2 f \left(\frac{n+1}{2}\right).
\]

\medskip
\noindent\textbf{Case 2: $n=2m$ for $m\ge 1$.}
Since $f$ is increasing, we get
\[
f \left(n+\frac{1}{2}\right)=f \left(2m+\frac{1}{2}\right)\le f(2m+1)=\sum_{k=0}^{2m}\nu_k.
\]
On the other hand,
\begin{align*}
f \left(\frac{n+1}{2}\right)
=f \left(m+\frac{1}{2}\right)
&=\int_0^1 \frac{1-r^{m+1/2}}{1-r} \dd \nu(r)\\
&=\sum_{k=0}^{m-1}\nu_k+\int_0^1 r^m\frac{1-r^{1/2}}{1-r} \dd \nu(r).
\end{align*}
Since, for every $r\in(0,1)$,
\[
\frac{1 - r^{1/2}}{1 - r} = \frac{1}{1 + r^{1/2}} \ge \frac{1}{2},
\]
we have
\[
f \left(\frac{n+1}{2}\right) \ge \sum_{k=0}^{m-1}\nu_k+\frac{1}{2}\int_0^1 r^m \dd \nu(r) = \sum_{k=0}^{m-1}\nu_k + \frac{1}{2} \nu_m.
\]
Moreover, using again that $(\nu_k)$ is decreasing, we have
\[
\sum_{k=m+1}^{2m}\nu_k \le \sum_{k=0}^{m-1}\nu_k.
\]
Therefore,
\begin{align*}
f \left(n+\frac{1}{2}\right)
&\le \sum_{k=0}^{2m}\nu_k =\sum_{k=0}^{m-1}\nu_k+\nu_m+\sum_{k=m+1}^{2m}\nu_k\\
&\le 2\sum_{k=0}^{m-1}\nu_k+\nu_m = 2\left(\sum_{k=0}^{m-1}\nu_k+\frac{1}{2} \nu_m\right)\\
&\le 2 f \left(m+\frac{1}{2}\right) =2 f \left(\frac{n+1}{2}\right).
\end{align*}

Combining the two cases, we obtain
\[
f \left(n+\frac{1}{2}\right)\le 2 f \left(\frac{n+1}{2}\right),
\qquad n\ge 1.
\]
Hence, for every $n\ge 1$,
\[
\frac{\omega_n}{\omega_{2n}}
=
\frac{f \left(n+\frac{1}{2}\right)}{f \left(\frac{n+1}{2}\right)}
\le 2.
\]
Together with the trivial case $n=0$, this shows that
\[
\omega_n\le 2 \omega_{2n}, \qquad n\in \mathbb N.
\]
Therefore $\omega$ is a $\widehat{\mathcal{D}}$-weight. This completes the proof.
\end{proof}

\section{The proof of Theorem \ref{thm:shimorin-kernel-characterization}}

The following theorem characterizes those Shimorin-type kernels that arise from
radial logarithmically subharmonic weights. In particular,
Theorem \ref{thm:shimorin-kernel-characterization} is an immediate consequence of
Theorem \ref{thm:new-shimorin-kernel-characterization} together with
Theorem \ref{representation of K_v}.

\begin{theorem}\label{thm:new-shimorin-kernel-characterization}
Let $\nu$ be a finite positive Borel measure on $[0,1]$ satisfying the Peláez–Rättyä–Wick condition.
Let $S_\nu$ be the corresponding Shimorin-type kernel.
Then there exists a radial logarithmically subharmonic weight on the unit disk
$
\omega(z)=\omega(|z|)
$
such that
\[
\int_0^1 r\omega(r) \dd r<\infty
\]
and
\[
K_\omega(z,\lambda)=S_\nu(z,\lambda), \qquad z,\lambda\in\mathbb D,
\]
if and only if there exists a logarithmically convex function
\[
h:(0,\infty)\to(0,\infty)
\]
such that
\[
\left(\int_0^1 \frac{1 - r^{n + 1}}{1 - r} \dd \nu(r)\right)
\left(\int_0^\infty e^{-nt}h(t) \dd t\right)=1,
\qquad n\in\mathbb N,
\]
and such that there exists $T>0$ with
\[
\sup_{t\ge T} e^t h(t)<\infty.
\]
\end{theorem}

For the proof of Theorem \ref{thm:new-shimorin-kernel-characterization}, we need some results on potential theory.

\begin{lemma}[{\cite[Corollary~2.5.2]{Ran95}}]\label{prop of subharmonic function}
Let \(u\) be a subharmonic function on a domain \(D\subset \mathbb{C}\), and assume that
$
u\not\equiv -\infty.
$
If \(\{z: |z - w| \le \rho \} \subset D\), then
\[
\frac{1}{2\pi}\int_{0}^{2\pi} u \left(w+\rho e^{i\theta}\right) \dd \theta>- \infty.
\]
\end{lemma}

\begin{lemma}\label{lem:radial-logsubharmonic-implies-convex}
Let $\varphi$ be a radial logarithmically subharmonic function on $\mathbb D$ with $\varphi \not\equiv 0$. Set
\[
u(z):=\log \varphi(z), \qquad z\in \mathbb D,
\]
where $u$ is allowed to take the value $-\infty$. Then the function
\[
g(t):=u(e^{-t/2})=\log \varphi(e^{-t/2}), \qquad t>0,
\]
is convex on $(0,\infty)$.
\end{lemma}

\begin{proof}
Since $\varphi$ is logarithmically subharmonic, the function
$
u=\log\varphi
$
is subharmonic on $\mathbb D$ and is not identically $-\infty$. Since $\varphi$ is radial,
$u$ is radial as well. Hence there exists a function, still denoted by $u$, such that
$u(z)=u(|z|),$ for any $z\in \mathbb D.$

We first claim that
\[
u(r)>-\infty,\qquad 0<r<1.
\]
Suppose, to the contrary, that there exists \(r_{0}\in(0,1)\) such that
\[
u(r_{0})=-\infty.
\]
Since \(u\) is radial, we have
\[
u(z)=u(|z|)=-\infty,\qquad |z|=r_{0}.
\]
Hence
\[
\frac{1}{2\pi}\int_{0}^{2\pi} u \left(r_{0}e^{i\theta}\right) \dd \theta=-\infty.
\]
On the other hand, \(u\) is subharmonic on \(\mathbb D\) and \(u\not\equiv -\infty\). Applying Lemma \ref{prop of subharmonic function} with \(D=\mathbb D\), \(w=0\), and \(\rho=r_{0}\), we obtain
\[
\frac{1}{2\pi}\int_{0}^{2\pi} u \left(r_{0}e^{i\theta}\right) \dd \theta>-\infty,
\]
which is a contradiction. Therefore, \text{for every } $r\in(0,1)$, we have
$
u(r)>-\infty.
$

We now prove that $g$ is convex. Fix $t_1,t_2>0$ and $\lambda\in(0,1)$, and set
\[
t:=(1-\lambda)t_1+\lambda t_2.
\]
Let
\[
r_1:=e^{-t_1/2},\qquad r_2:=e^{-t_2/2},\qquad r:=e^{-t/2}.
\]
Then
\[
r=r_1^{ 1-\lambda}r_2^{ \lambda}.
\]
Without loss of generality, assume $t_1<t_2$, so that $r_1>r_2$.

Consider the annulus
\[
A:=\{z\in\mathbb C:\ r_2<|z|<r_1\}.
\]
Define
\[
h(z):=
\frac{\log(r_1/|z|)}{\log(r_1/r_2)} u(r_2)
+
\frac{\log(|z|/r_2)}{\log(r_1/r_2)} u(r_1),
\qquad z\in A.
\]
Since $\log|z|$ is harmonic on $A$, the function $h$ is harmonic on $A$.
Moreover, $h$ is radial and satisfies
\[
h(\zeta)=u(r_1)\quad \text{for }|\zeta|=r_1,
\qquad
h(\zeta)=u(r_2)\quad \text{for }|\zeta|=r_2.
\]

Now $u-h$ is subharmonic on $A$, because $u$ is subharmonic and $h$ is harmonic.
Also, for every boundary point $\zeta\in\partial A$, we have
\[
\limsup_{A\ni z\to \zeta}(u(z)-h(z))\le 0.
\]
Indeed, if $|\zeta|=r_1$, then by radiality,
\[
\limsup_{A\ni z\to\zeta}u(z)\le u(r_1)=h(\zeta),
\]
and similarly, if $|\zeta|=r_2$, then
\[
\limsup_{A\ni z\to\zeta}u(z)\le u(r_2)=h(\zeta).
\]
Hence, by the maximum principle for subharmonic functions,
\[
u(z)\le h(z), \qquad z\in A.
\]
In particular, for every $z$ with $|z|=r$, $u(r)\le h(r)$.

Since $r=r_1^{1-\lambda}r_2^\lambda,$ we get
\[
\log\frac{r_1}{r}
=
\log\frac{r_1}{r_1^{1-\lambda}r_2^\lambda}
=
\lambda \log\frac{r_1}{r_2},
\]
and
\[
\log\frac{r}{r_2}
=
\log\frac{r_1^{1-\lambda}r_2^\lambda}{r_2}
=
(1-\lambda)\log\frac{r_1}{r_2}.
\]
Therefore
\[
h(r)
=
\lambda u(r_2)+(1-\lambda)u(r_1).
\]
Thus
\[
u(r)\le (1-\lambda)u(r_1)+\lambda u(r_2).
\]

Recalling that $r=e^{-t/2}$, $r_1=e^{-t_1/2}$, and $r_2=e^{-t_2/2}$, we obtain
\[
u(e^{-t/2})
\le
(1-\lambda)u(e^{-t_1/2})
+
\lambda u(e^{-t_2/2}),
\]
that is,
\[
g\big((1-\lambda)t_1+\lambda t_2\big)
\le
(1-\lambda)g(t_1)+\lambda g(t_2).
\]
Hence $g$ is convex on $(0,\infty)$.
\end{proof}

\begin{lemma}\label{lem:convex-implies-radial-logsubharmonic}
Let $h:(0,\infty)\to(0,\infty)$ be such that $\log h$ is convex on $(0,\infty)$.
Define
\[
\varphi(z)=\varphi(|z|):=\frac{h(-2\log|z|)}{|z|^2},
\qquad 0<|z|<1.
\]
Then $\varphi$ is logarithmically subharmonic on $\mathbb D\setminus\{0\}$.

If, in addition, $\varphi$ is bounded above in a neighborhood of $0$
(for example, if there exists $T>0$ such that
\[
\sup_{t\ge T} e^t h(t)<\infty),
\]
then $\varphi$ extends to a logarithmically subharmonic function on $\mathbb D$.
\end{lemma}

\begin{proof}
Since $h>0$, the function
\[
k(t):=\log h(t), \qquad t>0,
\]
is well defined. By assumption, $k$ is convex on $(0,\infty)$.

Define
\[
g(t):=k(t)+t=\log h(t)+t, \qquad t>0.
\]
Since the sum of a convex function and an affine function is convex, $g$ is convex on
$(0,\infty)$.

Now consider the right half-plane
\[
\mathbb H:=\{w\in\mathbb C: \operatorname{Re} w>0\},
\]
and define
\[
U(w):=g(\operatorname{Re} w), \qquad w\in\mathbb H.
\]
Then $U$ is subharmonic on $\mathbb H$.

%Fix $w_0=x_0+iy_0\in\mathbb H$, and let $\rho>0$ be such that
%$\overline{D(w_0,\rho)}\subset\mathbb H$. Then $x_0>\rho$, and for every $\theta\in[0,2\pi]$,
%\[
%\operatorname{Re}(w_0+\rho e^{i\theta})=x_0+\rho\cos\theta>0.
%\]
%Therefore
%\[
%\frac1{2\pi}\int_0^{2\pi} U(w_0+\rho e^{i\theta}) \dd \theta
%=
%\frac1{2\pi}\int_0^{2\pi} g(x_0+\rho\cos\theta) \dd \theta.
%\]
%Since $g$ is convex and
%\[
%\frac1{2\pi}\int_0^{2\pi}(x_0+\rho\cos\theta) \dd \theta=x_0,
%\]
%Jensen's inequality gives
%\[
%g(x_0)
%\le
%\frac1{2\pi}\int_0^{2\pi} g(x_0+\rho\cos\theta) \dd \theta.
%\]
%Hence
%\[
%U(w_0)\le \frac1{2\pi}\int_0^{2\pi} U(w_0+\rho e^{i\theta}) \dd \theta.
%\]
%Thus $U$ satisfies the sub-mean inequality, and therefore $U$ is subharmonic on $\mathbb H$.

Next fix $z_0\in\mathbb D\setminus\{0\}$. Choose a simply connected neighborhood
$V\subset\mathbb D\setminus\{0\}$ of $z_0$ on which a holomorphic branch of the logarithm
exists; denote it by $\Log z$. Define
\[
F(z):=-2\Log z, \qquad z\in V.
\]
Then $F$ is holomorphic on $V$, and
\[
\operatorname{Re} F(z)=\operatorname{Re}(-2\Log z)=-2\log|z|>0, \qquad z\in V,
\]
because $0<|z|<1$. Hence $F(V)\subset\mathbb H$.

Since $U$ is subharmonic on $\mathbb H$ and $F$ is holomorphic, the composition $U\circ F$ is subharmonic on $V$.
On the other hand, for $z\in V$,
\[
(U\circ F)(z)
=
g(\operatorname{Re} F(z))
=
g(-2\log|z|).
\]
By the definition of $g$,
\[
g(-2\log|z|)
=
\log h(-2\log|z|)-2\log|z|.
\]
Since
\[
\varphi(z)=\frac{h(-2\log|z|)}{|z|^2},
\]
it follows that
\[
(U\circ F)(z)=\log\varphi(z).
\]
Therefore $\log\varphi$ is subharmonic on $V$. Since $z_0$ was arbitrary, we conclude that $\log\varphi$
is subharmonic on $\mathbb D\setminus\{0\}$. Equivalently, $\varphi$ is logarithmically
subharmonic on $\mathbb D\setminus\{0\}$.

For the last assertion, assume that $\varphi$ is bounded above in a neighborhood of $0$.
Then $\log\varphi$ is bounded above in a punctured neighborhood of $0$ and is subharmonic
on $\mathbb D\setminus\{0\}$. By the removable singularity theorem for subharmonic functions,
the function
\[
\widetilde u(z):=
\begin{cases}
\log\varphi(z), & z\neq 0,\\[4pt]
\limsup\limits_{\zeta\to 0}\log\varphi(\zeta), & z=0,
\end{cases}
\]
is subharmonic on $\mathbb D$. Setting
\[
\widetilde\varphi(z):=e^{\widetilde u(z)},
\]
we obtain a logarithmically subharmonic extension of $\varphi$ to the whole disk.
This completes the proof.
\end{proof}

\begin{proof}[Proof of Theorem \ref{thm:new-shimorin-kernel-characterization}]
\textbf{Necessity.}
Assume that there exists a radial logarithmically subharmonic weight
$\omega$ on $\mathbb D$ such that
\[
\int_0^1 r\omega(r) \dd r<\infty
\]
and
\[
K_\omega(z,\lambda)=S_\nu(z,\lambda), \qquad z,\lambda\in\mathbb D.
\]
By the coefficient formulas for the kernels $K_\omega$ and $S_\nu$, it follows that
for every $n\in\mathbb N$,
\[
\left(\int_0^1 \frac{1 - r^{n + 1}}{1 - r} \dd \nu(r)\right)
\left(2\int_0^1 r^{2n+1}\omega(r) \dd r\right)=1.
\]
Define the function $h:(0,\infty)\to(0,\infty)$ by
\[
h(t):=e^{-t}\omega(e^{-t/2}), \qquad t>0.
\]
By Lemma~\ref{lem:radial-logsubharmonic-implies-convex}, the function $t\mapsto \log\omega(e^{-t/2})$ is convex on $(0,\infty)$. Hence
\[
\log h(t)=\log\omega(e^{-t/2})-t
\]
is also convex on $(0,\infty)$, because subtracting an affine function preserves convexity.
Thus $h$ is logarithmically convex.

Next, by the change of variables $r=e^{-t/2}$, we obtain for each $n\in\mathbb N$,
\[
\int_0^\infty e^{-nt}h(t) \dd t =\int_0^\infty e^{-nt}e^{-t}\omega(e^{-t/2}) \dd t = 2\int_0^1 r^{2n+1}\omega(r) \dd r.
\]
Therefore
\[
\left(\int_0^1 \frac{1 - r^{n + 1}}{1 - r} \dd \nu(r)\right)
\left(\int_0^\infty e^{-nt}h(t) \dd t\right)=1,
\qquad n\in\mathbb N.
\]

It remains to verify the growth condition. Since $\log\omega$ is subharmonic on $\mathbb D$,
it is locally bounded above near $0$. Hence $\omega$ is bounded above in some neighborhood
of $0$. Therefore, for all sufficiently large $t$,
\[
e^t h(t)=\omega(e^{-t/2})
\]
is bounded, that is, there exists $T>0$ such that
\[
\sup_{t\ge T} e^t h(t)<\infty.
\]

\medskip

\noindent
\textbf{Sufficiency.}
Assume that there exists a logarithmically convex function
$h:(0,\infty)\to(0,\infty)$ such that
\[
\left(\int_0^1 \frac{1 - r^{n + 1}}{1 - r} \dd \nu(r)\right)
\left(\int_0^\infty e^{-nt}h(t) \dd t\right)=1,
\qquad n\in\mathbb N,
\]
and such that
\[
\sup_{t\ge T} e^t h(t)<\infty
\]
for some $T>0$.

Define
\[
\omega(z):=\frac{h(-2\log|z|)}{|z|^2},
\qquad 0<|z|<1.
\]
By Lemma~\ref{lem:convex-implies-radial-logsubharmonic}, the function $\omega$
is logarithmically subharmonic on $\mathbb D\setminus\{0\}$, and the growth condition
\[
\sup_{t\ge T}e^t h(t)<\infty
\]
implies that $\omega$ is bounded above near $0$. Hence $\omega$ extends to a
logarithmically subharmonic function on $\mathbb D$. Since the original function
on $\mathbb D\setminus\{0\}$ is radial, the extension is radial as well.
Since $h(t)>0$ for all $t>0$, we have $\omega(r)>0$ for all $0<r<1$.
Hence $\omega([r,1))>0$ for every $0<r<1$, so $L_a^2(\omega)$ is a reproducing kernel Hilbert space.
We next show that
\[
\int_0^1 r\omega(r) \dd r<\infty.
\]
Taking $n=0$, we get
\[
\nu([0,1])\int_0^\infty h(t) \dd t=1.
\]
Since $\nu([0,1])>0$, it follows that
\[
\int_0^\infty h(t) \dd t=\frac{1}{\nu([0,1])}<\infty.
\]
Using again the change of variables $t=-2\log r$, we obtain
\[
\int_0^1 r\omega(r) \dd r
=\int_0^1 \frac{h(-2\log r)}{r} \dd r
=\frac{1}{2}\int_0^\infty h(t) \dd t
=\frac{1}{2 \nu([0,1])}<\infty.
\]

Finally, for every $n\in\mathbb N$,
\begin{align*}
2\int_0^1 r^{2n+1}\omega(r) \dd r
&=2\int_0^1 r^{2n-1} h(-2\log r) \dd r\\
&=\int_0^\infty e^{-nt}h(t) \dd t.
\end{align*}
Hence, for each $n\in\mathbb N$,
\[
\left(\int_0^1 \frac{1 - r^{n + 1}}{1 - r} \dd \nu(r)\right)
\left(2\int_0^1 r^{2n+1}\omega(r) \dd r\right)=1.
\]
By the coefficient formulas for $K_\omega$ and $S_\nu$, this implies
\[
K_\omega(z,\lambda)=S_\nu(z,\lambda), \qquad z,\lambda\in\mathbb D.
\]
This completes the proof.
\end{proof}

\end{document}